\documentclass[12pt,reqno]{amsart}
\usepackage{graphicx, amsmath, amsthm, amsfonts, enumerate, amssymb, mathrsfs, color}
\usepackage{fullpage}
\usepackage{bbm}

\def\eff{\text{\rm eff}}
\def\hom{\text{\rm hom}}

\def\Port{\mathcal P_{\bot}}
\def\P{\mathcal P}

\def\e{\varepsilon}

\DeclareMathOperator*{\dom}{\mathrm{dom}}

\DeclareMathOperator*{\diag}{\mathrm{diag}}
\DeclareMathOperator*{\ran}{\mathrm{ran}}

\newtheorem{theorem}{Theorem}[section]
\newtheorem{proposition}[theorem]{Proposition}

\newtheorem*{conjecture}{Conjecture}

\newtheorem{definition}[theorem]{Definition}
\newtheorem{lemma}[theorem]{Lemma}

\newtheorem{remark}[theorem]{Remark}

\usepackage[pdftex]{hyperref}

\allowdisplaybreaks

\begin{document}
\title[Laplacians on graph-convergent manifolds]{Norm-resolvent convergence for Neumann Laplacians on manifolds thinning to graphs}
\author{Kirill D. Cherednichenko}
\address{Department of Mathematical Sciences, University of Bath, Claverton Down, Bath, BA2 7AY, United Kingdom}
\email{k.cherednichenko@bath.ac.uk}
\author{Yulia Yu. Ershova}
\address{Department of Mathematics, Texas A\&M University,
	College Station, TX 77843-3368, USA}
\email{julija.ershova@gmail.com}
\author{Alexander V. Kiselev}
\address{Department of Mathematical Sciences, University of Bath, Claverton Down, Bath, BA2 7AY, United Kingdom}
\email{alexander.v.kiselev@gmail.com}


\begin{abstract}
Norm-resolvent convergence with an order-sharp error estimate is established for Neumann Laplacians on thin domains in $\mathbb{R}^d,$
$d\ge2,$ converging to metric graphs in the limit of vanishing thickness parameter in the ``resonant" case. The vertex matching conditions of the limiting quantum graph are revealed as being closely related to those
of  $\delta'$ type.
\end{abstract}

\maketitle

\section{Introduction}

In the book \cite{Post}, see also references therein, in particular \cite{Exner,KuchmentZeng,KuchmentZeng2004}, Neumann Laplacians $A_\e$ on thin manifolds, converging to metric graphs  as $\varepsilon\to0$, were studied, see, e.g., Fig.\,\ref{fig:kaplya1}. Here $\varepsilon$ represents the ``thickness'' of the manifolds in those parts where they converge to the graph edges. The named works attacked the question of spectral (and, in the case of \cite{Post}, norm-resolvent) convergence of such partial differential operators to a graph Laplacian with certain matching conditions at the graph vertices. This latter ordinary differential operator is introduced as follows. Denoting by $E$ the set of edges $e$ of the limiting graph $G,$ each $e\in E$ can be identified with the interval $[0,l_e],$ where $l_e$ is the length of $e$. We use notation $L^2(e)$  for the associated Hilbert space $L^2(0,l_e).$ Similarly, we denote $L^2(G):= \oplus_e L^2(e)$. The graph Laplacian $A_G$ is generated by the differential expression $-u'',$ $u\in H^2(e)$ on each edge $e\in E$ separately (see \cite{Kuchment2} for details), subject to the vertex conditions discussed below.

It was proved in \cite{Exner, KuchmentZeng, KuchmentZeng2004} that, within any compact $K\in \mathbb{C},$ the spectra of $A_\e$ converge in the Hausdorff sense to the spectrum of a graph Laplacian $A_G$. In the book \cite{Post}, the claimed convergence was enhanced to the norm-resolvent type, with an explicit control of the error as $O(\varepsilon^\gamma),$ where $\gamma>0$ depends on whether the ambient space is two-dimensional. The matching conditions at the vertices of the limiting graph turn out to be either of:
\begin{itemize}
  \item[(i)] Kirchhoff (i.e. standard), if the vertex volumes are decaying, as $\varepsilon\to0,$ faster than the edge volumes;
  \item[(ii)] ``Resonant", which can be equivalently described in terms of $\delta$-type matching conditions with coupling constants proportional to the spectral parameter $z,$ see \cite{KuchmentZeng, Exner}, if the vertex and edge volumes are of the same order;
  \item[(iii)] ``Dirichlet-decoupled" (i.e., the graph Laplacian becomes completely decoupled), if the vertex volumes vanish slower than the edge ones.
\end{itemize}

We also refer the reader to the frequently overlooked papers \cite{MPP, Pavlov_2008}, where an alternative approach to the asymptotic analysis of thin networks, aimed at capturing the resonant and scattering features, is developed, see also the references therein. In the present paper we consider the case of scalar PDEs --- for recent progress in the analysis of PDE systems, set in the context of thin elastic structures (albeit not forming networks), see \cite{ChV, ChVZ, BChVZ}.

In the present paper we are primarily interested in the most non-trivial resonant case (ii). We provide a straightforward, alternative to \cite{Post}, proof of the fact that the Neumann Laplacians $A_\e$ in this case converge in norm-resolvent sense to a linear operator acting in the Hilbert space $L^2(G)\oplus \mathbb{C}^N$, where $N$ is the number of vertices. The mentioned limiting operator is in fact the one first pointed out in \cite{KuchmentZeng2004} as a self-adjoint operator whose spectrum coincides with the Hausdorff limit of spectra for the family $A_\e$.

On the technical side, our approach can be seen as  a modification of the one developed by us in \cite{ChEK, ChKVZ}, see also \cite{Physics, GrandePreuve, CherKis, CherKisSilva}. We specifically point out that the framework originally developed for high-contrast homogenisation admits a natural generalisation to setups where the ``contrast" is achieved by purely geometric means, including (but not limited to) thin networks  as in the present paper. This seems to widen significantly the range of dimension-reduction type models that are amenable to this kind of analysis and thus to establish a transparent connection between previously unrelated physical contexts, providing for a possibility to develop new types of media in materials science.

We obtain a better error bound than \cite[Section 6.7]{Post} (and without visiting the ``plumber's shop" of \cite{RSch}). Our estimate in the planar case is $O(\e|\log\e|)$ and in the case of $\mathbb{R}^d,$ $d\ge3,$ it is $O(\e).$ 
Unlike that of \cite{Post}, our method does not allow us to study the full set of asymptotic regimes in (iii).
This is due to the fact that the argument of our paper \cite{ChEK} is based on the Dirichlet-to-Neumann (DN) machinery. There is a possibility to modify the approach by invoking Neumann-to-Dirichlet maps instead, which would have two advantages: one could consider all rates of vertex volume decay in (iii), and certain geometric smoothness requirements could be somewhat relaxed.
Nevertheless, in this paper we stick with the DN version of the approach, in order to align the exposition with that of \cite{ChEK}. The alternative strategy will be followed up elsewhere, both in the present context and in the setting of \cite{ChEK}.

The above results of course imply the Hausdorff spectral convergence, at the same time yielding a sharp estimate on its rate. Moreover, in contrast to \cite{Post}, our approach allows one to consider ``high-frequency" regimes, i.e., setups where the spectral parameter (which in the wave propagation context may represent the square of the frequency) is no longer constrained to a compact set but is still constrained by some negative power of the small parameter $\e$. (In non-dimensional terms this corresponds to the wavelength being of the order of some positive power of $\e$). This leads to a sequence of ``effective", dimensionally reduced, models of the thin structure, which are sequentially applicable for a set of (asymptotic) frequency intervals. The complexity of the dimension reduction process for these models increases along the sequence. While an initial result regarding the high-frequency situation is presented below, we postpone the full analysis to a future publication.

Alongside the high-frequency analysis, yet another sequence of models will be revealed by a version of the same argument. This corresponds to the transition from the resonant regime to the Dirichet-decoupled one and will allow us to reconcile the asymptotic analysis of \cite{Post} with that of \cite{Pavlov_2008}, by introducing ``transitional" models of increasing complexity.  In these transitional regimes, from the point of view of \cite{Post} one gets arbitrarily close to the Dirichlet-decoupled situation, see (iii) above, whereas from the point of view of \cite{Pavlov_2008} (where the vertex volumes do not decay at all) one faces a highly non-trivial picture of resonant scattering.

Aiming at better clarity, in the present paper we restrict ourselves to the case where (without much loss of generality) the edge subdomains are assumed straight and uniformly thin, whereas the vertex subdomains are smooth with, possibly, the exception of the points where they meet the edges (see Section \ref{setup_section} for further details). These assumptions appear very natural, in view of the possible examples shown in Fig.\,\ref{fig:kaplya1}, \ref{fig:kaplya}.

\begin{figure}[h!]
	\begin{center}
		\def\svgwidth{.8\textwidth}
		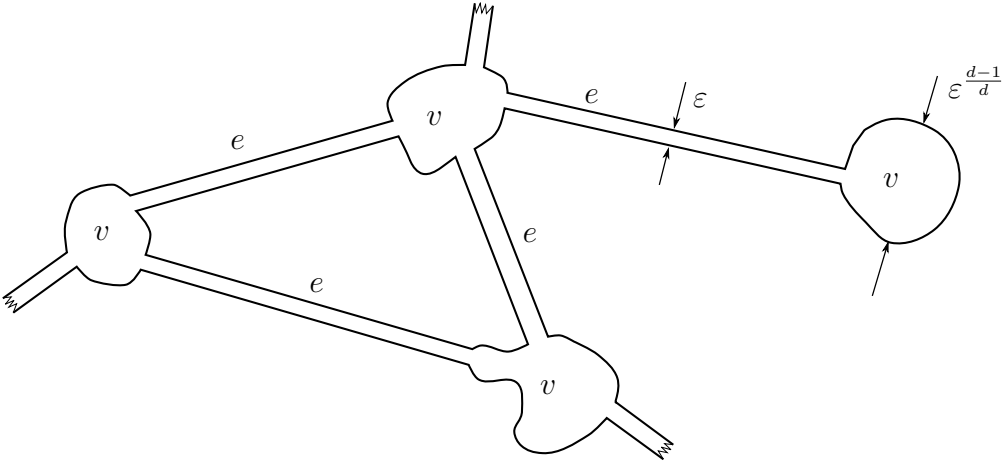
	\end{center}
	\caption{An example of a thin network in the case of a resonant scaling of the edge thickness and vertex diameter.  The (non-dimensional) parameter $\varepsilon$ is the ratio of the actual size to the wavelength.  As $\varepsilon\to0,$ the structure converges to a graph; the related notation is introduced at the start of Section~\ref{setup_section}. \label{fig:kaplya1}}
\end{figure}

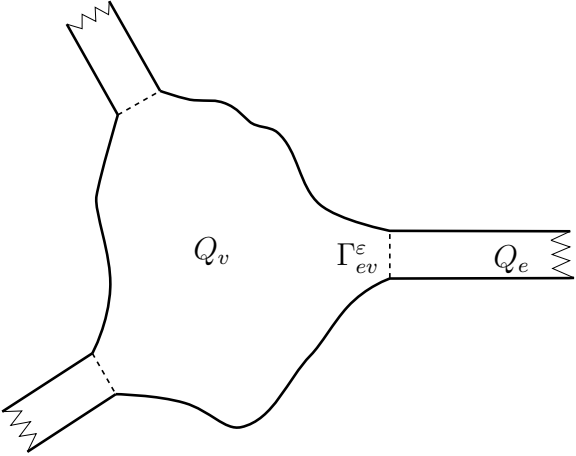
\begin{figure}[h!]
	\begin{center}
		\def\svgwidth{.5\textwidth}
\begingroup%
  \makeatletter%
  \providecommand\color[2][]{%
    \errmessage{(Inkscape) Color is used for the text in Inkscape, but the package 'color.sty' is not loaded}%
    \renewcommand\color[2][]{}%
  }%
  \providecommand\transparent[1]{%
    \errmessage{(Inkscape) Transparency is used (non-zero) for the text in Inkscape, but the package 'transparent.sty' is not loaded}%
    \renewcommand\transparent[1]{}%
  }%
  \providecommand\rotatebox[2]{#2}%
  \newcommand*\fsize{\dimexpr\f@size pt\relax}%
  \newcommand*\lineheight[1]{\fontsize{\fsize}{#1\fsize}\selectfont}%
  \ifx\svgwidth\undefined%
    \setlength{\unitlength}{178.68518961bp}%
    \ifx\svgscale\undefined%
      \relax%
    \else%
      \setlength{\unitlength}{\unitlength * \real{\svgscale}}%
    \fi%
  \else%
    \setlength{\unitlength}{\svgwidth}%
  \fi%
  \global\let\svgwidth\undefined%
  \global\let\svgscale\undefined%
  \makeatother%
  \begin{picture}(1,0.72904221)%
    \lineheight{1}%
    \setlength\tabcolsep{0pt}%
    \put(0,0){\includegraphics[width=\unitlength,page=1]{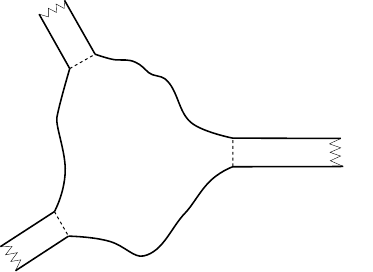}}%
    \put(0.53922316,0.30270456){\makebox(0,0)[lt]{\lineheight{1.25}\smash{\begin{tabular}[t]{l}$\Gamma^\e_{ev}$\end{tabular}}}}%
    \put(0.7923637,0.29965467){\makebox(0,0)[lt]{\lineheight{1.25}\smash{\begin{tabular}[t]{l}$Q_e$\end{tabular}}}}%
    \put(0.30895671,0.31185434){\makebox(0,0)[lt]{\lineheight{1.25}\smash{\begin{tabular}[t]{l}$Q_v$\end{tabular}}}}%
  \end{picture}%
\endgroup%

	\end{center}
	\caption{An example of a vertex subdomain $Q_v$ and an edge subdomain $Q_e$ separated by a contact plate $\Gamma_v^\varepsilon$ of linear size $\varepsilon.$ The volumes of $Q_v$ and $Q_e$ are assumed to be of the same order $\varepsilon^{d-1},$ which is also the area of the contact plate.  As in Fig.\,\ref{fig:kaplya}, the (non-dimensional) parameter $\varepsilon$ is the ratio of the actual size to the wavelength. \label{fig:kaplya}}
\end{figure}

\section{Problem setup and preliminaries}
\label{setup_section}


In what follows, we consider a prototypical setup only, which already presents all the challenges appearing in the general case. We also refer the reader to \cite{Post}, where the most general setup is carefully introduced.

For the limiting graph,
the following notation will be used: the metric graph $G$ will be identified with the set of edges $E$, so each individual edge is denoted by $e\in E$ and is associated with an interval $[0,l_e]$. We denote by $V$ the set of graph vertices and treat each $v\in V$ as the set of edge endpoints meeting at $v$. The graph $G$ is assumed to be oriented throughout.

Proceeding to the setup for the Neumann Laplacian on a thin graph-convergent structure, let a connected domain $Q$ be the union of the ``vertex" part $Q_V$ and the ``edge" part $Q_E$, where $Q_E$ will be assumed to be a finite collection of $\e$-thin cylinders (for $d=2,$ rectangular boxes), $Q_E=\cup_e Q_e$ of lengths $l_e.$ For each $e$, the domain $Q_e$ is assumed to be, up to a linear spatial transform (i.e. shift and rotation), defined by (for $d=2$)
$$
Q_e=\bigl\{x\in \mathbb R^2: x_1\in (0,l_e), x_2\in (0,\e)\bigr\}=(0, l_e)\times(0,\varepsilon).
$$
In dimensions $d\ge3,$ the set $Q_e$ is defined as the direct product of the interval $(0, l_e)$ and a smooth cross-section $Q_e^{\rm c}$ of volume $\varepsilon^{d-1}.$

It is further assumed that $Q_V=\cup_{v} Q_v$, where each of the disjoint domains $Q_v$ is connected
and has piecewise smooth boundary $\partial Q_v,$ which can be decomposed as $\partial Q_v=\widetilde{\Gamma}^\varepsilon_v\cup \Gamma^\e_v$, where $\Gamma^\e_v$ is the interface between the vertex and edge domains and $\widetilde{\Gamma}^\varepsilon_v$ is the remaining part of $\partial Q_v.$ Henceforth, we often drop the dependence on $\varepsilon$ in the notation for those geometric elements where it can be removed by a rescaling or is otherwise irrelevant in the analysis to follow.

The contact part  $\Gamma^\e_v$ is further decomposed into a union of flat plates (for $d=2,$ straight segments): $\Gamma^\e_v=\cup\Gamma_{ev}^\e$. Here the union is taken over all edge domains $Q_e$ connected to $Q_v$ so that $\overline{Q}_e\cap \overline{Q}_v=\Gamma_{ev}^\e$. In what follows, we will refer to the segments $\Gamma_{ev}^\e$ as \emph{contact plates}. Since operators of the Zaremba (or mixed) boundary value problem \cite{Zaremba} will be used below, we further require that the contact plates $\Gamma_{ev}^\e$ meet $\widetilde{\Gamma}^\varepsilon_v$ at angles strictly less than $\pi$, see \cite{Brown,Brown2} for further details. We will further assume that the curves $\widetilde{\Gamma}^\varepsilon_v$ are smooth.

Furthermore, we assume that for all vertices $v$ the domains $Q_v$
have volumes of order $\varepsilon^{d-1}$ and are obtained by suitable transformations of an $\e$-independent domain $Q_{v}^0$ containing the origin. More precisely, we assume that for all $v$ and $\varepsilon>0$ one has
\begin{equation}
Q_v=
	\e^{\frac{d-1}{d}}{\mathcal T}_v^\varepsilon (Q_v^0)+b_v^\varepsilon,
\label{Qrep}
\end{equation}
where $b_v^\varepsilon\in\mathbb R^d,$ and the domain 
 ${\mathcal T}_v^\varepsilon (Q_v^0)$ is obtained from $Q_v^0$ by a suitable homeomorphism ${\mathcal T}_v^\varepsilon$ that: a) preserves volume; b) maintains the angles between the contact plates and the adjacent parts of the boundary to be uniformly less than $\pi,$ cf. \cite{Brown,Zaremba}.

 For example, if $Q_v$ is star-shaped with respect to the origin\footnote{A domain $Q_v$ is said to be star-shaped with respect to $x_v\in Q_v$ if for all $x\in Q_v$ one has $\{x_v+t(x-x_v), t\in[0,1]\}\subset Q_v.$}, the deformation ${\mathcal T}_v^\varepsilon$ can be constructed by ``cut-and-glue surgery" applied to $Q_v^0$ followed by a suitable ``radial" scaling, as follows. 
 The domain $Q_v^0$ is first split into several conical (sectorial for $d=2$) subdomains by making cuts alongside the cone ``generatrices" (``radii" for $d=2$) from the origin to the boundary-points (end-points for $d=2$) of all subsets of $\partial Q_v^0$ pertaining to the contact plates of $Q_v.$ Each cone is then transformed by a suitable change of variable so that: A) the images of the cone bases pertaining to the contact plates of $Q_v$ have linear size $\varepsilon^{1/d}$; B) the image of the complementary part of $Q_v^0$ is reattached to the lateral boundaries of the cones in A. Note that this procedure in general changes the domain volume by a small quantity.
In order to restore the volume, a suitable ``radial" scaling (centred at the origin) is applied, ensuring that it does not affect any of the mentioned conical subsets.

 We remark that the representation (\ref{Qrep}) guarantees that we are in the ``resonant'' (or ``borderline") case of \cite{Exner,KuchmentZeng}, i.e., that the volumes of the contact plates are proportional to the volumes of vertex domains.

The requirement that the homeomorphism ${\mathcal T}_v^\varepsilon$ in \eqref{Qrep} be volume-preserving is not essential and can be removed, as long as the ratios of the volumes of the vertex domains to the volumes of the contact plates are bounded below. In this case the corresponding final convergence estimates, similar to those we derive below (see Theorem \ref{thm:main_fin}), will in general depend on the said ratios, as is evident from the formula \eqref{eq:Rhom}.




On the domain $Q$ we consider a family of self-adjoint operators $A_\e$ defined by their sesquilinear forms and corresponding to the differential expression $-\Delta$ subject to Neumann boundary conditions. Other types of boundary conditions can be considered as well, including those of Robin and Dirichlet \cite{Grieser, MPP}, in which case the lower edge of the spectrum is $\varepsilon$-dependent, which in the context of homogenisation corresponds to the so-called high-frequency regime or, using the terminology of
M.~S.~Birman, to ``the neighbourhood of the edge of a gap" \cite{Birman_2004, CKP}.

Consider $\sigma>0$ and an arbitrary non-negative function $R=R(\varepsilon),$ $\varepsilon\in(0,1)$ such that
$R(\varepsilon)\varepsilon|\log \varepsilon|^\gamma\to0$ as $\varepsilon\to0,$ where $\gamma=1$ for $d=2$ and $\gamma=0$ otherwise. In what follows, we will deal with the family of resolvents $(A_\e-z)^{-1}$. We shall always assume that $z\in \mathbb C$ is separated from the spectrum of the original operator family. In particular cases, we will assume that
either $z$ is constrained to the $\varepsilon$-growing compact set
$$
K_\sigma^\varepsilon:=B_{R(\varepsilon)}(0)\setminus\bigl\{z\in \mathbb C |\ \text{dist}(z, \mathbb R)\geq \sigma\},
$$
where $B_{R(\e)}(0)$ denotes the ball of radius $R(\e)$ centered at the origin,
or
$$
z\in K_\sigma:=\bigl\{z\in \mathbb C |\ z\in K \text{ a compact set in } \mathbb C,\ \text{dist}(z, \mathbb R)\geq \sigma\bigr\}.
$$
In particular, for all $\varepsilon$ small enough, one has $K_\sigma\subset K_\sigma^\e.$
After we have established the operator-norm asymptotics of $(A_\e-z)^{-1}$ for $z\in K_\sigma,$ the result is extended by analyticity to a compact set $K_\sigma^{\text{ext}}$ whose distance to the spectrum of the leading order of the asymptotics is bounded below by $\sigma$, and the same statement holds in the case of $K_\sigma^\e.$

Proceeding similarly to, e.g., \cite{Friedlander} in the related area of critical-contrast homogenisation and facilitated by the abstract framework of \cite{Ryzh_spec},   we consider $A_\e$ as operators of transmission problems (see \cite{Schechter} and references therein) relative to the \emph{internal boundary} $\Gamma^\varepsilon:=\cup_{e,v} \Gamma_{ev}^\varepsilon.$
The transmission problem is formulated as, given a function $f\in L^2(Q)$, finding the weak solution $u\in L^2(Q)$ of the boundary value problem
\begin{equation}
	\label{eq:transmissionBVP}
\begin{cases}
-\Delta u(x)-zu(x)=f(x), & x\in Q_V \text{\ or\ } x\in Q_E,\\[0.3em]
u_v(x)=u_e(x),\quad \dfrac{\partial u_v}{\partial n}- \dfrac{\partial u_e}{\partial n}=0& \text{ on } \Gamma_{ev}^\varepsilon,\\[0.4em]
\dfrac{\partial u}{\partial n}  =0 & \text{ on } \partial Q.
\end{cases}
\end{equation}
Here $u_v:=u|_{Q_v}$, $u_e:=u|_{Q_e}$ for all admissible $e$ and $v$ (i.e. when $\Gamma_{ev}^\varepsilon\neq\varnothing$), and $n$ represents the exterior normal on $\partial Q$ and the ``edge-inward'' normal (i.e., directed from $Q_v$ to $Q_e$) on any of contact plates $\Gamma_{ev}^\varepsilon.$
By a classical argument, the solution of the above problem is shown to be equal to $(A_\e-z)^{-1}f$. Note that the  boundary value problem for the Neumann Laplacian on $Q$ is given by the first and third lines in \eqref{eq:transmissionBVP}. While the second line is, strictly speaking, redundant, it proves important in order to view the problem as a transmission one relative to $\Gamma^\e.$

It remains to be seen that the linear operator of the transmission problem \eqref{eq:transmissionBVP} defined via the technique of \cite{Ryzh_spec}, which we briefly recall below, is the same operator $A_\e$; the proof of this fact follows easily by combining   \cite{Ryzh_spec} 
and the main estimate of \cite{Schechter}.

Following the approach of \cite{Ryzh_spec} (cf. \cite{BehrndtLanger2007, BMNW2008} and references therein for alternative approaches), which is based on the ideas of the classical Birman-Kre\u\i n-Vi\v sik theory (see \cite{Birman,Krein,Vishik}), the linear operator of the transmission boundary value problem is introduced as follows. Let $\mathcal H:=L^2(\Gamma^\varepsilon)=\oplus_{e,v} L^2(\Gamma_{ev}^\varepsilon)$, and consider the harmonic lift operators $\Pi_V$ and $\Pi_E$ defined on $\phi\in\mathcal H$ via
\begin{equation*}
	\begin{aligned}
\Pi_V \phi:=u_\phi,\ \ {\rm where}\ \  \begin{cases}
\Delta u_\phi=0,\ \ u_\phi\in L^2(Q_V),&\\
u_\phi|_{\Gamma^\varepsilon} = \phi,
\end{cases}\\[1.0em]
\Pi_E \phi:=u_\phi,\ \ {\rm where}\ \  \begin{cases}
\Delta u_\phi=0,\ \ u_\phi\in L^2(Q_E),&\\
u_\phi|_{\Gamma^\varepsilon} = \phi,
\end{cases}
\end{aligned}
\end{equation*}
subject to Neumann boundary conditions on $\partial Q$. These operators are first defined on $\phi\in C^2_0(\Gamma^\varepsilon)$, in which case the corresponding solutions $u_\phi$ can be seen as classical \cite{Zaremba}. The results of \cite{Brown} allow one to extend both harmonic lifts to bounded (in fact, compact) operators on $\mathcal H$, in which case $u_\phi$ are to be treated as distributional solutions of the respective boundary value problems. The solution operator $\Pi:\mathcal H\mapsto L^{2}(Q)=L^{2}(Q_V)\oplus L^{2}(Q_E)$ is defined as follows:
$$
\Pi \phi:= \Pi_V \phi \oplus \Pi_E \phi.
$$

Consider the self-adjoint operator family $A_0$ (choosing not to reflect the $\varepsilon$-dependence in the notation) to be the Dirichlet decoupling of the operator family $A_\e$, i.e., the operator of the boundary value problem on both $Q_V$ and $Q_E$, where the Dirichlet boundary conditions are imposed on $\Gamma^\varepsilon$ together with Neumann boundary conditions on $\partial Q$. The operator $A_0$ is generated by the same differential expression as $A_\e$. Clearly, one has $A_0=A_0^{V}\oplus A_0^{E}$ relative to the orthogonal decomposition $L^2(Q)=L^2(Q_V)\oplus L^2(Q_E)$; all three operators $A_0, A_0^{V}$ and $A_0^{E}$ are self-adjoint and positive-definite. Moreover, by \cite{Brown,Denzler1} there exists a bounded inverse $A_0^{-1}$. Note that $\dom A_0\cap\ran \Pi=\varnothing,$ see \cite{Ryzh_spec}.

Furthermore, denoting by $\widetilde \Gamma_0^{V}$ and $\tilde \Gamma_0^{E}$ the left inverse of $\Pi_{V}$ and $\Pi_{E},$ respectively, one introduces the trace operator $\widetilde\Gamma_0^{V}$ (respectively, $\widetilde\Gamma_0^{E}$) as the null extension of $\widetilde \Gamma_0^{V}$ (respectively, $\widetilde \Gamma_0^{V}$) to the domain $\dom A_0^{V}\dotplus \ran \Pi_{V}$ (respectively, $\dom A_0^{E}\dotplus \ran \Pi_{E}.$) In the same way we introduce the operator $\tilde \Gamma_0$ and its null extension $\Gamma_0$ to the domain $\dom A_0 \dotplus \ran \Pi$.

The solution operators $S_z^V$, $S_z^E$ of the boundary value problems
\begin{align*}
\begin{cases}
-\Delta u_\phi-z u_\phi=0,\ \ u_\phi\in \dom A_0^V\dotplus \ran\Pi_V,&\\[0.3em]
\Gamma_0^V u_\phi = \phi,&
\end{cases}&\\[0.4em]
\begin{cases}
-\Delta u_\phi-z u_\phi=0,\ \  \ u_\phi\in \dom A_0^E\dotplus \ran\Pi_E,&\\[0.3em]
\Gamma_0^E u_\phi = \phi&
\end{cases}&
\end{align*}
are defined as linear mappings from $\phi$ to $u_\phi$, respectively. These operators are bounded from $L_2(\Gamma^\varepsilon)$ to $L_2(Q_V)$ and $L_2(Q_E)$, respectively, and admit the following representations:
\begin{equation*}
S_z^E=(1-z (A_0^E)^{-1})^{-1}\Pi_E,\quad S_z^V=(1-z (A_0^V)^{-1})^{-1}\Pi_V.
\end{equation*}
The solution operator $S_z$ from $L^2(\Gamma^\varepsilon)$ to $L^2(Q_V)\oplus L^2(Q_E)$ is now defined as $S_z=S_z^V\oplus S_z^E$; it admits the representation $S_z=(1-z (A_0)^{-1})^{-1}\Pi$ and is bounded.

Having introduced orthogonal projections $P_V$ and $P_E$ from $L^2(Q)$ onto $L^2(Q_V)$ and $L^2(Q_E)$, respectively, one has the obvious identities
\begin{equation*}
S_z^V=P_V S_z,\quad S_z^E=P_E S_z,
\qquad
\Pi_V=P_V \Pi,\quad \Pi_E=P_E \Pi.
\end{equation*}

Fix self-adjoint (and, in general, unbounded) operators $\Lambda^E$, $\Lambda^V$ defined on domains $\dom \Lambda^E,$ $\dom \Lambda^V\subset L^2(\Gamma^\varepsilon)$ (in what follows these operators will be chosen as DN maps of Zaremba problems on $Q_E$ and $Q_V$, respectively, and well-defined on $H^1(\Gamma^\varepsilon)$, where $H^1(\Gamma^\varepsilon)$ is the standard Sobolev space pertaining to the internal boundary $\Gamma^\varepsilon$). Still following \cite{Ryzh_spec}, we define the ``second boundary operators'' $\Gamma_1^E$ and $\Gamma_1^V$ to be linear operators on the domains
\begin{equation*}
\dom \Gamma_1^E:=\dom A_0^E\dotplus \Pi_E \dom \Lambda^E, \quad \dom \Gamma_1^V:=\dom A_0^V\dotplus \Pi_V \dom \Lambda^V.
\end{equation*}
The action of $\Gamma_1^{E (V)}$ is set by:
\begin{equation*}
\Gamma_1^E:\ (A_0^E)^{-1}f\dotplus \Pi_E\phi \mapsto \Pi_E^* f+\Lambda^E \phi,
\quad
\Gamma_1^V:\ (A_0^V)^{-1}f\dotplus \Pi_V\phi \mapsto \Pi_V^* f+\Lambda^V \phi
\end{equation*}
for all $f\in L^2(Q_E)$, $\phi\in \dom \Lambda^E$ and $f\in L^2(Q_V)$, $\phi\in \dom \Lambda^V$, respectively.

Alongside $\Gamma_1^{E (V)}$, introduce a self-adjoint operator $\Lambda$ on $\dom \Lambda\subset \mathcal H$ and then the following ``boundary" operator
$
\Gamma_1:
$
$$
\dom \Gamma_1:= \dom A_0^{-1}\dotplus \Pi \dom \Lambda,
$$
$$
\Gamma_1:\ A_0^{-1}f\dotplus \Pi \phi \mapsto \Pi^* f + \Lambda\phi\qquad \forall\,f\in L^2(Q),\,
\phi\in\dom \Lambda.
$$
We remark that the operators $\Gamma_1, \Gamma_1^{E (V)}$ thus defined are assumed to be neither closed nor indeed closable.

In our setup, we make the following concrete choice of the operators $\Lambda^{E(V)}$: in what follows, they are the DN maps pertaining to the components $Q_E$ and $Q_V$, respectively. More precisely, for the problem
\begin{gather*}
\Delta u_\phi=0, \quad u_\phi\in L^2(Q_E),\\[0.1cm]
u_\phi|_{\Gamma^\varepsilon} = \phi,\quad \partial_n u_\phi|_{\partial Q}=0,
\end{gather*}
the operator $\Lambda^E$ maps the boundary values $\phi$ of $u_\phi$ to the negative traces of its normal derivative\footnote{This definition is inspired by \cite{Ryzh_spec}. Note that the operator thus defined is negative the classical $DN$ map of e.g. \cite{Friedlander}.}
\[
\partial_n u_\phi|_{\Gamma^\varepsilon}:=\nabla u\cdot n|_{\Gamma^\varepsilon},
\]
where $n=-n_E$ is as above the ``edge-inward'' normal. This operator is well defined by its sesquilinear form as a self-adjoint operator on $L^2(\Gamma^\varepsilon)$ (see, e.g., \cite{Agranovich,Grubb}), and $\dom \Lambda^E\supset H^1(\Gamma^\varepsilon)$ by \cite{Brown}.

On the vertex part $Q_V$ we consider the problem
\begin{equation*}
\begin{gathered}
\Delta u_\phi=0, \quad u_\phi\in L^2(Q_V),\\[0.2em]
u_\phi|_{\Gamma^\varepsilon} = \phi, \quad \partial_n u_\phi|_{\partial Q}=0,
\end{gathered}
\end{equation*}
and define $\Lambda^V$ as the operator mapping the boundary values $\phi$ of $u_\phi$ to the negative traces of its normal derivative $-\partial_n u_\phi|_{\Gamma^\varepsilon} $, where $n=n_V$ is again the ``edge-inward'' normal. The self-adjointness of $\Lambda^V$ on $\dom\Lambda^V\supset H^1(\Gamma^\varepsilon)$ follows by an unchanged argument.

Finally we introduce the operator $\Lambda$ which on $\phi\in H^1(\Gamma^\varepsilon)$ is the sum $\Lambda\phi = \Lambda^V\phi + \Lambda^E \phi$. It is also a self-adjoint operator on $\dom \Lambda \supset H^1(\Gamma^\varepsilon)$. This can be ascertained either by the argument of \cite{Schechter}, in which case it is defined as the inverse of a compact self-adjoint operator on the orthogonal complement to constants $L^2(\Gamma^\varepsilon)\ominus\{c\pmb 1\}$, extended to $\{c\pmb 1\}$ by zero, or, alternatively, from its definition by a closed sesquilinear form.

The choice of $\Lambda^{E (V)}$ made above allows us to consider $\Gamma_1$  on the domain $\dom A_0 \dotplus \Pi \dom\Lambda$. One then writes \cite{Ryzh_spec} the second Green identity in the following form:
\begin{equation*}
\langle Au,v \rangle_{L^2(Q)} - \langle u, Av \rangle_{L^2(Q)} = \langle \Gamma_1 u, \Gamma_0 v \rangle_{L^2(\Gamma^\varepsilon)} - \langle \Gamma_0 u, \Gamma_1 v \rangle_{L^2(\Gamma^\varepsilon)}
\end{equation*}
for all $u,v\in \dom \Gamma_1=\dom A_0\dotplus \Pi \dom\Lambda$, where the operator $A$ is the null extension (see \cite{Ryzhov_later}) of the operator $A_0$ onto $\dom \Gamma_1$. Thus the triple $(\mathcal H, \Gamma_0, \Gamma_1)$ is closely related to a boundary quasi-triple of \cite{BehrndtLanger2007} (see also \cite{BehrndtRohleder2015}) for the transmission problem considered; cf. \cite{BMNW2008} for an alternative approach.

The calculation of $\Pi^*$ in \cite{Ryzh_spec} shows that $\Pi^*=\Gamma_1 A_0^{-1}$ and therefore $\Gamma_1$ as introduced above acts as follows:
$$
\Gamma_1:\ u=P_E u + P_V u\,\mapsto\,\partial_n P_e u|_{\Gamma^\varepsilon} - \partial_n P_V u|_{\Gamma^\varepsilon},
$$
where $P_E$ and $P_V$ are the orthogonal projections of $L^2(Q)$ onto $L^2(Q_E)$, $L^2(Q_V)$, respectively. Therefore, transmission problem at hand (at least, formally so far) corresponds to the interface (or ``matching") condition $\Gamma_1 u=0$.

\begin{definition}[\cite{Ryzh_spec}]\label{defn:3}
The operator-valued function $M(z)$ defined on the domain $\dom \Lambda$ for
$z\in \rho(A_0)$ (and in particular, for $z\in K_\sigma^\varepsilon$) by the formula
\begin{equation}\label{defn:M-function}
  M(z) \phi = \Gamma_1 S_z \phi = \Gamma_1 (1-z A_0^{-1})^{-1}\Pi \phi
\end{equation}
is referred to as the $M$-function of the problem \eqref{eq:transmissionBVP}.
\end{definition}

The following result of \cite{Ryzh_spec} summarises the properties of the $M-$function which we will need in what follows.

\begin{proposition}[\cite{Ryzh_spec}, Theorem 3.3.]\label{prop:M}

1. One has the following representation:
\begin{equation}\label{eq:M_representation}
  M(z)=\Lambda + z \Pi^* (1-z A_0^{-1})^{-1}\Pi, \quad z\in \rho(A_0).
\end{equation}

2. $M(z)$ is an analytic operator-function with values in the set of closed operators in $L_2(\Gamma^\varepsilon)$ densely defined on the $z$-independent domain $\dom \Lambda$.

3. For $z,\zeta \in\rho (A_0)$ the operator $M(z) - M(\zeta)$ is bounded and
$$
M(z) - M(\zeta) = (z - \zeta)S^*_{\bar z} S_\zeta
$$
In particular, $\Im M(z) = (\Im z) S^{*}_{\bar z} S_{\bar z}$ and $(M(z))^*  = M(\bar z)$.

4. For $u_z \in \ker (A - zI) \cap \{ \dom A_0  \dotplus \Pi \dom \Lambda \}$ the following formula holds:
$$
M(z)\Gamma_0 u_z = \Gamma_1 u_z.
$$
\end{proposition}

Alongside $M(z),$ we define $M_V(z)$ and $M_E(z),$ which pertain to the vertex $Q_V$ and edge $Q_E$ parts of the domain $Q,$ respectively, by the formulae
\begin{equation}\label{defn:M-stiffsoft}
\begin{gathered}
M_V(z) \phi = \Gamma^V_1 S_z^V \phi = \Gamma_1^V\bigl(1-z (A_0^V)^{-1}\bigr)^{-1}\Pi_V \phi,
\\ M_E(z) \phi = \Gamma^E_1 S_z^E \phi = \Gamma_1^E \bigl(1-z (A_0^E)^{-1}\bigr)^{-1}\Pi_E \phi.
\end{gathered}
\end{equation}
As before, in the notation we suppress the dependence on the parameter $\varepsilon$ for brevity.

The value of the fact that for $\phi\in H^1(\Gamma^\varepsilon)$ one has $M(z)\phi=M_E(z)\phi+M_V(z)\phi$ is clear: in contrast to $A_\e,$ which cannot be additively decomposed into ``independent'' terms pertaining to the vertex and edge parts of the medium $Q$ owing to the transmission interface conditions on $\Gamma^\varepsilon,$ the $M$-function is additive (see, e.g., \cite{ChEK, CherKisSilva}, where this property was observed and exploited in the related settings of homogenisation and scattering, respectively). In what follows, we will observe that the resolvent $(A_\e-z)^{-1}$ can be expressed in terms of $M(z)$ via a version of the celebrated Kre\u\i n formula, thus reducing the asymptotic analysis of the resolvent to that of the corresponding $M$-function (see, e.g., \cite{AGW2014,BehrndtLanger2007} for alternative approaches to derivation of the Kre\u\i n formula in our setting).

Alongside the transmission problem \eqref{eq:transmissionBVP}, the boundary conditions of which can be now (so far, formally) represented as $u\in \dom A_0\dotplus \Pi L^2(\Gamma^\varepsilon)$, $\Gamma_1 u=0$, in what follows we will require a wider class of problems of this type. This class is formally given by the transmission conditions
$$
u\in \dom A\equiv \dom A_0\dotplus \Pi L^2(\Gamma^\varepsilon),\quad \beta_0 \Gamma_0 u+\beta_1 \Gamma_1 u =0,
$$
where $\beta_1$ is a bounded operator on $L^2(\Gamma^\varepsilon)$ and $\beta_0$ is a linear operator defined on the domain $\dom \beta_0\supset \dom \Lambda$.

In general, the operator $\beta_0\Gamma_0+\beta_1 \Gamma_1$ is not defined on the domain $\dom A$. This problem is being taken care of by the following assumption, which will be satisfied throughout:
$$
\beta_0+\beta_1 \Lambda \text{ defined on $\dom\Lambda$ is closable in } \mathcal H.
$$
We remark that by Proposition \ref{prop:M} the operators $\beta_0+\beta_1 M(z)$ are then closable for all $z\in \rho(A_0)$, and the domains of their closures coincide with $\dom \overline{\beta_0+\beta_1 \Lambda}$.

For any $f\in H,$ $\phi\in \dom \Lambda$, the equality
$$
(\beta_0\Gamma_0+\beta_1\Gamma_1)(A_0^{-1}f+\Pi \phi)=\beta_1 \Pi^* f + (\beta_0+\beta_1\Lambda)\phi
$$
shows that the operator $\beta_0\Gamma_0+\beta_1\Gamma_1$ is correctly defined on $A_0^{-1}H\dotplus \Pi \dom \Lambda\subset \dom A$. Denoting $\mathcal B:=\overline{\beta_0+\beta_1\Lambda}$ with the domain $\dom\mathcal B\supset \dom \Lambda$, one checks that $H_{\mathcal B}:=A_0^{-1}H\dotplus \Pi\dom\mathcal B$ is a Hilbert space with respect to the norm
$$
\|u\|^2_{\mathcal B}:=\|f\|^2_H+\|\phi\|^2_{\mathcal H}+\|\mathcal B\phi\|^2_{\mathcal H}, \qquad
u=A_0^{-1}f+\Pi\phi.
$$
It is then proved \cite[Lemma 4.1]{Ryzh_spec} that $\beta_0\Gamma_0+\beta_1\Gamma_1$ extends to a bounded operator from $H_{\mathcal B}$ to $\mathcal H.$ For the sake of convenience, same notation $\beta_0\Gamma_0+\beta_1\Gamma_1$ is preserved for this extension.

We will make use of the following version of the celebrated Kre\u\i n formula.
\begin{proposition}[\cite{Ryzh_spec},Theorem 5.1]\label{prop:operator}
Let $z\in\rho(A_0)$ be such that the operator $\overline{\beta_0+\beta_1 M(z)}$ defined on $\dom \mathcal B$ is boundedly invertible. Then
\begin{equation}\label{eq:Krein_general}
R_{\beta_0,\beta_1}(z):=(A_0-z)^{-1}+S_z Q_{\beta_0,\beta_1}(z)S_{\bar z}^*, \text{ where } Q_{\beta_0,\beta_1}:=-(\overline{\beta_0+\beta_1 M(z)})^{-1}\beta_1,
\end{equation}
is the resolvent of a closed densely defined operator $A_{\beta_0,\beta_1}$ with the domain
$$
\dom A_{\beta_0,\beta_1}=\{u\in H_{\mathcal B}| (\beta_0\Gamma_0+\beta_1\Gamma_1) u=0\}=\ker(\beta_0\Gamma_0+\beta_1\Gamma_1).
$$
\end{proposition}

In particular, the (self-adjoint) operator of the transmission problem \eqref{eq:transmissionBVP}, which corresponds to the choice $\beta_0=0, \beta_1=I$, admits the following characterisation in terms of its resolvent:
\begin{equation}\label{eq:Krein_transmission}
R_{0,I}(z)=(A_0-z)^{-1}-S_z M^{-1}(z) S_{\bar z}^*.
\end{equation}
In this case, one clearly has $H_{\mathcal B}=A_0^{-1}H\dotplus \Pi \dom \Lambda$ and $\dom A_{0,I}=\{u\in H_{\mathcal B}| \Gamma_1 u=0\},$ which, together with the discussion at the beginning of this section, yields $A_{0,I}=A_\e$.

We remark that the operators $\beta_0$ and $\beta_1$ above can be assumed $z$-dependent, as this change does not impact the corresponding proofs of \cite{Ryzh_spec}. In this case however, the corresponding operator-function $R_{\beta_0,\beta_1}(z)$ is shown to be the resolvent of a $z$-dependent operator family. Within the self-adjoint setup of the present paper, $R_{\beta_0,\beta_1}(z)$ is guaranteed to represent a generalised resolvent in the sense of \cite{Naimark1940,Naimark1943,Strauss}.

\section{Auxiliary estimates}\label{sec:aux}

In this section we collect a number of auxiliary statements required in our proof of the main result. 

We start with the analysis of the operators $\Pi_V$ and $S_z^V$ introduced in Section 2. First we note that each of these operators admits a decomposition into an orthogonal sum over $N$ vertex domains $\{Q_v\}$ of $Q$. 
It therefore suffices to consider a single vertex domain $Q_v$ (we recall for readers' convenience that the volume of this domain is assumed to be decaying with $\e\to 0$). Its boundary $\partial Q_v$ contains a disjoint set of straight segments belonging to the internal boundary $\Gamma^\varepsilon,$ which are, in line with what has been said above, denoted as $\Gamma_{ev}^\varepsilon$; the union of the latter is $\Gamma_v^{\e}$.

The decoupled operator $A_0$ has $L^2(Q_v)$ as its invariant subspace. We will denote by $A_0^{(v)}$ its self-adjoint restriction,
$A_0^{(v)}:=A_0|_{L^2(Q_v)}$. By construction, the operator $A_0^{(v)}$ is the Laplacian with the so-called Zaremba, or mixed Neumann-Dirichlet, boundary condition \cite{Zaremba,Polterovich}. More precisely, it is subject to the Dirichlet boundary condition on $\Gamma^{\e}_v$ and to Neumann boundary condition on its complement $\widetilde{\Gamma}_v^\e$. Clearly this operator is boundedly invertible; moreover, the following statement holds.

\begin{proposition}[see \cite{Denzler1,Denzler2}]\label{prop:Denzler}
There exists a constant $C>0$ such that for all $\varepsilon$ one has
$$
\bigl\|(A_0^{(v)})^{-1}\bigr\|\leq C|Q_v|\left\{\begin{array}{ll}
	|\log \e|,&\ d=2,\\[0.4em]
	\varepsilon^{2-d},&\ d\ge3\end{array}\right.\asymp\e|\log \e|^\gamma,
$$
where, as before, $\gamma=1$ for $d=2$ and $\gamma=0$ otherwise.
\end{proposition}

\begin{remark}
The above proposition holds under more general conditions than those we impose. Namely, the domain $Q_v$ is only required to be Lipschitz and no conditions whatsoever are imposed on the geometry of the set $\Gamma_v^\e$.
\end{remark}

Next, we turn our attention to the solution operator $S_z^{(v)}:=S_z^V|_{\Gamma^{\e}_v}$ and the corresponding harmonic lift $\Pi^{(v)}:=\Pi_V|_{\Gamma^{\e}_v}$. The two are clearly related by the formula
$$
S_z^{(v)}=\bigl(1-z(A_0^{(v)})^{-1}\bigr)^{-1}\Pi^{(v)}.
$$
In order to bound the norm of $\Pi^{(v)}$, we can follow, e.g., the following approach. First, consider the corresponding Zaremba problem on $Q_{v}^0$. We proceed by relating the norm of the corresponding Poisson operator to the least Steklov eigenvalue of the bi-Laplacian, following the blueprint of \cite{Kuttler}, based in turn on Fichera's duality principle, see \cite{Fichera}. Since the boundary of $Q_{v}^0$ is non-smooth, in doing so we follow the generalisations developed in \cite{Gazzola_non_smooth,Auchmuty}, with obvious modifications required when passing from the Dirichlet to Zaremba setup. The estimate for the said Steklov eigenvalue is then taken from the norm of the compact embedding of $H^2(Q_{v}^0)$ to the traces of normal derivatives on the contact plates, see e.g. \cite{Gazzola}. Rescaling back to $Q_v$, we obtain the following auxiliary result.



\begin{lemma}\label{lemma:Pi_estimate}
There exists $C>0$ such that $\|\Pi^{(v)}\|\leq C$
for all $\varepsilon.$
\end{lemma}

By Proposition \ref{prop:Denzler}, the above lemma
yields the following estimate for the solution operator $S_z^{(v)}.$
\begin{lemma}
For $\varepsilon\in(0,1),$ uniformly in $z\in K_\sigma^\varepsilon$ (and, in particular, for $z\in K_\sigma$) one has
$$
S_z^{(v)}=
	\bigl(1+O\left(|z|\e|\log \e|^\gamma\right)\bigr)\Pi^{(v)}
	=\Pi^{(v)} + O(|z|\e|\log \e|^\gamma),
$$
where the error bounds are understood in the uniform operator norm topology.
\end{lemma}

Our next step is the analysis of the ``part'' of the DN map $\Lambda^V$ pertaining to the vertex domain $Q_v$. We will denote by $\Lambda_{v}^V$ its self-adjoint restriction
$\Lambda^V|_{L^2(\Gamma_v^\e)}=:\Lambda_{v}^V.$ 

First, we note that the spectrum of $\Lambda_v^V$ (which can be termed as the Steklov spectrum of the sloshing problem pertaining to $A_0^{(v)}$, see \cite{Levitin_et_al2018}) is discrete and accumulates to negative infinity. The point $\lambda_1=0$ is the least (by absolute value) Steklov eigenvalue with $\psi_v=|\Gamma^{\e}_v|^{-1/2} \pmb 1|_{\Gamma^{\e}_v}$ being the corresponding eigenvector. For the second eigenvalue $\lambda_2$ one has the following estimate, see, e.g., \cite{Agranovich} and references therein.
\begin{lemma}\label{lemma:Lambda_estimate}
	There exists $C>0$ such that
$
|\lambda_2|\ge C
	\e^{-1}.
$
\end{lemma}

Introduce the $N$-dimensional orthogonal projection
$$
\P:=\sum_{v}\langle \cdot, \psi_v\rangle \psi_v,
$$
define $\Port:=1-\P$
and consider the operator $\Port M(z) \Port,$ which is well defined since $\P \dom\Lambda \subset \dom \Lambda.$ By a straightforward estimate for sesquilinear forms, see \cite[Section 3.2]{ChEK}, and taking into account \eqref{eq:M_representation} applied to $M_V(z)$ and combined with Proposition \ref{prop:Denzler} as well as Lemmata \ref{lemma:Pi_estimate}, \ref{lemma:Lambda_estimate}, one has the following statement.
\begin{lemma}\label{lemma:Mort}
	There exists $C>0$ such that for all $\varepsilon\in(0,1)$ one has
$$
\bigl\|(\Port M(z) \Port)^{-1}\bigr\| \le C \e,\qquad  z\in K_\sigma^\varepsilon,
$$
where the operator $\Port M(z) \Port$ is considered as a linear (unbounded) operator in $\Port\mathcal H$.
\end{lemma}

We conclude this section by noting that, alternatively, one can derive the bound on the Poisson operator $\Pi^{(v)}$ claimed in Lemma \ref{lemma:Pi_estimate} by employing the scaling property of the Dirichlet-to-Neumann map $\Lambda^V_v,$ similar to the argument of \cite{Agranovich} referenced above, combined with a standard estimate on the solutions to the classical Neumann problem.

\section{Norm-resolvent asymptotics}

We will make use of the Kre\u\i n formula \eqref{eq:Krein_transmission} to obtain a norm-resolvent asymptotics of the family $A_\e$. In doing so, we will compute the asymptotics of $M^{-1}(z)$ based on a Schur-Frobenius type inversion formula, having first rewritten $M(z)$ as a $2\times 2$ operator matrix relative to the orthogonal decomposition of the Hilbert space $\mathcal H=\P \mathcal{H}\oplus\Port \mathcal{H}$. In the study of operator matrices, we rely upon the material of \cite{Tretter}, see also references therein.


The operator $M(z)$ admits the block matrix representation
$$
M(z)=\begin{pmatrix}
       \mathbb A & \mathbb B \\
       \mathbb E & \mathbb D
     \end{pmatrix} \text{ with } \mathbb A, \mathbb B, \mathbb E \text{ bounded.}
$$
For the inversion of $M(z)$ we then use the Schur-Frobenius inversion formula \cite[Theorem 2.3.3]{Tretter}
\begin{equation}\label{eq:SchurF}
\overline{\begin{pmatrix}
       \mathbb A & \mathbb B \\
       \mathbb E & \mathbb D
     \end{pmatrix}}^{-1}=
\begin{pmatrix}
  \mathbb  A^{-1}+\overline{\mathbb  A^{-1}\mathbb B}\overline{\mathcal{S}}^{-1}\mathbb E\mathbb A^{-1} & -\overline{\mathbb A^{-1}\mathbb B}\overline{\mathcal{S}}^{-1} \\
  -\overline{\mathcal{S}}^{-1}\mathbb E\mathbb A^{-1} & \overline{\mathcal{S}}^{-1}
\end{pmatrix}
\text{ with } \mathcal{S}:=\mathbb D-\mathbb E \mathbb A^{-1} \mathbb B.
\end{equation}

Note that by Proposition \ref{prop:M}, one has $\Im M(z) = (\Im z) S^{*}_{\bar z} S_{\bar z}.$
 Moreover, since $S_z=(1-zA_0^{-1})^{-1}\Pi,$ one has
 \[
 S^{*}_{\bar z} S_{\bar z}=\Pi^*(1-zA_0^{-1})^{-1}(1-{\bar z}A_0^{-1})^{-1}\Pi,
 \]
 and therefore, for some  constants $c_1, c_2>0,$
 \[
\langle S^{*}_{\bar z} S_{\bar z}\P\phi, \P\phi\rangle_{\mathcal H}=\bigl\Vert(1-{\bar z}A_0^{-1})^{-1}\Pi\P\phi\bigr\Vert^2\ge c_1\Vert\Pi\P\phi\Vert^2\ge c_1\Vert\Pi_V\P\phi\Vert^2\ge c_2\Vert\P\phi\Vert_{\mathcal H}^2
 \]
 for all $\phi\in{\mathcal H}, z\in K_\sigma^\varepsilon,$
 where we have used the fact that the operator $A_0$ is bounded below by a positive constant. It follows that  ${\mathbb A}^{-1}=(\P M(z)\P)^{-1}$ is boundedly invertible.

Proceeding exactly as in \cite{ChEK} based on the estimate provided by Lemma \ref{lemma:Mort} which now reads
$$
\bigl\|\mathbb D^{-1}\bigr\|\leq C\e,
$$
we use $\mathcal{S}^{-1}=(I-\mathbb D^{-1}\mathbb E\mathbb A^{-1}\mathbb B)^{-1} \mathbb D^{-1}$ to obtain $\mathcal S^{-1}=O(\e)$.

Returning to \eqref{eq:SchurF}, one gets
\begin{equation}\label{eq:estimatefortheorem}
M(z)^{-1}=\begin{pmatrix}
       \mathbb A &\mathbb  B \\
       \mathbb E &\mathbb  D
     \end{pmatrix}^{-1}=
\begin{pmatrix}
 \mathbb  A^{-1}& 0\\
  0 & 0
\end{pmatrix}
+O(\e)
\end{equation}
with a uniform estimate for the remainder term. Comparing our result with \eqref{eq:Krein_general} of Proposition \ref{prop:operator} with $\beta_0:=\Port$ and $\beta_1:=\P$, one arrives at the following

\begin{theorem}\label{thm:NRA}
There exists $C>0$ such that for all $\varepsilon\in(0,1)$ and $z\in K_\sigma^\varepsilon$ (in partucular, for all
$z\in K_\sigma$)
one has the estimate
$$
\bigl\|(A_\e-z)^{-1}-(A_{\beta_0,\beta_1}-z)^{-1}\bigr\|\leq C\e
$$
for a universal constant $C$ and $\beta_0=\Port$, $\beta_1=\P$, where the operator $A_{\beta_0,\beta_1}$ is defined in Proposition \ref{prop:operator}.
\end{theorem}
\begin{proof}
The proof is identical to that of \cite[Theorem 3.1]{ChEK}; we include it here for the sake of completeness.

For the resolvent $({A}_\e-z)^{-1}$ the formula (\ref{eq:Krein_transmission}) is applicable, in which for $M(z)^{-1}$ we use (\ref{eq:estimatefortheorem}). As for the resolvent
$\bigl(A_{\Port,\P}-z\bigr)^{-1},$ Proposition \ref{prop:operator} with $\beta_0=\Port,$ $\beta_1=\P$
is clearly applicable.
Moreover, for this choice of $\beta_0,$ $\beta_1,$ the operator
$$
Q_{\Port, \P}(z)=-\bigl(\overline{\Port+\P M(z)}\bigr)^{-1}\P
$$
in (\ref{eq:Krein_general}) is easily computable ({\it e.g.}, by the Schur-Frobenius inversion formula of \cite{Tretter}, see \eqref{eq:SchurF})\footnote{We remark that $\Port+\P M(z)$ is triangular (${\mathbb A}=\P M(z)\P,$ ${\mathbb B}=\P M(z)\Port,$ ${\mathbb E}=0,$ ${\mathbb D}=I$ in \eqref{eq:SchurF}) with respect to the decomposition $\mathcal H=\P\mathcal H\oplus \Port \mathcal H$.}, yielding
\begin{equation}
Q_{\Port, \P}(z)=-\P\bigl(\P M(z) \P\bigr)^{-1}\P,
\label{Qadd}
\end{equation}
and the claim follows.
\end{proof}

Already the estimate of Theorem \ref{thm:NRA} establishes norm-resolvent convergence of the family $A_\e$ to an operator which by \eqref{Qadd} is a relative (i.e. with respect the difference of resolvents) finite-rank perturbation of the decoupled operator $A_0$. However, in the case $z\in K_\sigma,$ which we will assume henceforth, it is possible to obtain a further simplification of this answer, relating the leading-order asymptotic term to a self-adjoint operator on the limiting metric graph. This procedure follows the blueprint of our paper \cite{ChEK}. We next briefly outline the related argument. For the case of $z$ not constrained to a compact, a similar argument yields a sequence of dimensionally reduced models, as mentioned in the Introduction.

Note first that $(A_0-z)^{-1}=(A_0^V-z)^{-1}\oplus (A_0^E-z)^{-1}$ is easily analysed. Indeed, by Proposition \ref{prop:Denzler} one has
\begin{equation}
(A_0^V-z)^{-1}=
	O\bigl(\e|\log \e|^\gamma\bigr).
\label{A0Vest}
\end{equation}
Furthermore, the operator $(A_0^E-z)^{-1},$ by separation of variables\footnote{This is the only place where we use the assumption about the geometric shape of the edge parts $Q_e$ of the thin structure. This can be generalised to the setup of \cite{Post}, allowing for curvature and non-uniform thickness, leading to Laplace-Beltrami operators on the edges of the limiting graph.}, is $O(\varepsilon^2)$-close to the Dirichlet Laplacian on the space
$$
H_G := \oplus_v L^2([0,l_e]\times \pmb 1_\e),
$$
where $\pmb 1_e:= \e^{-(d-1)/2} \pmb 1 $ is the normalised constant function in the variable transverse to  the edge $e$. The operator $(A_0-z)^{-1}$ is therefore 
close, uniformly in $z\in K_\sigma$, to an operator that is unitary equivalent to the resolvent of $A_0^G$, where $A_0^G$ is the Dirichlet-decoupled graph Laplacian pertaining to the graph $G$. The related error estimates are the same as in (\ref{A0Vest}). The finite-dimensional second term on the right-hand side of \eqref{Qadd} is therefore expected to encode the matching conditions at the vertices of the limiting graph $G$. In order to see this, one passes over to the generalised resolvent $R_\e(z):=P_E (A_\e-z)^{-1} P_E$, which is shown to admit the following asymptotics.

\begin{theorem}\label{thm:NRA_gen}
The operator family $R_\e(z)$ admits the following asymptotics in the operator-norm topology for $z\in K_\sigma$:
$$
R_\e(z)-R_\eff(z)=O(\e),
$$
where $R_\eff(z)$ is the solution operator for the following spectral BVP on the edge domain $Q_E$:
\begin{equation}
\label{eq:BVPBz}
\begin{aligned}
-&\Delta u-zu=f, \quad f\in L^2(Q_E),\\[0.2em]
&\beta_0(z)\Gamma_0^E u+ \beta_1\Gamma_1^E u=0,
\end{aligned}
\end{equation}
with $\beta_0(z)=\Port-\P B(z)\P$, $B(z):=-M_V(z)$ and $\beta_1=\P$.

The boundary condition in (\ref{eq:BVPBz}) can be written in the more conventional form
$$
\Port u|_{\Gamma^\varepsilon} =0, \quad \P \partial_n u = \P B(z)\P u\bigl|_{\Gamma^\varepsilon}.
$$
Equivalently,
$$
R_\e(z)-\bigl(A^E_{\Port-\P B(z)\P,\P}-z\bigr)^{-1}=O(\e),
$$
where $A^E_{\Port-\P B(z)\P,\P},$ for any fixed $z,$ is the operator in $L^2(Q_E)$ defined by Proposition \ref{prop:operator} relative to the triple $(\mathcal H, \Pi_E, \Lambda^E)$,  where the term ``triple" is understood in the sense of \cite{Ryzh_spec}. This operator is maximal anti-dissipative for $z\in\mathbb C_+$ and maximal dissipative for $z\in \mathbb C_-,$ see \cite{Strauss}.

\end{theorem}

The \emph{proof} of the theorem follows immediately from Theorem \ref{thm:NRA}, see \cite[Theorem 3.6]{ChEK} together with the observation that
$$
\P M(z)\P = \P M_E(z) \P + \P M_V(z) \P.
$$

The next step of our argument is to introduce the truncated\footnote{In what follows we consistently supply the (finite-dimensional)  ``truncated'' spaces and operators pertaining to them by the breve overscript.} (reduced) boundary space $\breve{\mathcal H}$ in order to make all the ingredients finite-dimensional.

We put $\breve {\mathcal H}:=\P \mathcal H$ (noting that in our setup $\breve{\mathcal H}$ is $N$-dimensional, where $N$ is the number of vertices, see Section \ref{setup_section}). Introduce the truncated Poisson operator on $\breve{\mathcal H}$ by $\breve{\Pi}_E:=\Pi_E|_{\breve{\mathcal H}}$ and the truncated DN map $\breve{\Lambda}^E:=\P\Lambda^E\vert_{\breve {\mathcal H}}.$
Then the following statement holds.

\begin{proposition}[\cite{ChEK}, Theorem 3.7]
\label{thm:NRA_gen_truncated}

1. The formula
\begin{equation}\label{eq:gen_res_trunc}
R_\eff(z)=\bigl(A_0^E-z\bigr)^{-1}-\breve{S}_z^E\bigl(\breve{M}_E(z)-\P B(z)\P\bigr)^{-1} (\breve{S}_{\bar z}^E)^*
\end{equation}
holds, where $\breve{S}_z^E$ is the solution operator of the problem
$$
\begin{aligned}
-\Delta u_\phi-z u_\phi&=0,\ \ u_\phi\in \dom A_0^E\dotplus \ran\breve{\Pi}_E,\\[0.1cm]
\Gamma_0^E u_\phi&= \phi, \quad \phi\in \breve{\mathcal H},
\end{aligned}
$$
and $\breve{M}_E$ is the $M$-operator defined in accordance with \eqref{defn:M-function}, \eqref{defn:M-stiffsoft} relative to the triple $(\breve{\mathcal H}, \breve{\Pi}_E, \breve{\Lambda}^E)$.

2. The ``effective'' generalised resolvent $R_\eff(z)$ is represented as the generalised resolvent of the problem
\begin{equation*}
\begin{gathered}
-\Delta u-zu =f, \quad f\in L^2(Q_E),\quad u\in \dom A_0^E\dotplus \ran \breve{\Pi}_E,\\[0.4em]
\P\partial_n u\bigr|_{\Gamma^\varepsilon} = \P B(z)\P u\bigr|_{\Gamma^\varepsilon}.
\end{gathered}
\end{equation*}

3. The triple $(\breve{\mathcal H},\breve{\Gamma^\varepsilon}_0^E, \breve{\Gamma^\varepsilon}_1^E)$ is the classical boundary triple \cite{Gor,DM} for the operator $A_{\max}$ defined by the differential expression $-\Delta$ on the domain $\dom A_{\max}=\dom A_0^E\dotplus \ran\breve{\Pi}_E.$ Here $\breve{\Gamma^\varepsilon}_0^E$ and $\breve{\Gamma^\varepsilon}_1^E$ are defined on $\dom A_{\max}$ as the operator of the boundary trace on $\Gamma$ and
$\P\partial_n u,$ respectively.
\end{proposition}

We now consider the operator $\P B(z)\P$ in \eqref{eq:gen_res_trunc}; since $B=-M_V$ by definition, we invoke the estimates derived in Section 3  to obtain
$$
\P B\P=-\P\Lambda^V \P-z \P \Pi_V^*\Pi_V \P + O(\e|\log \e|^\gamma) = -z \breve{\Pi}_V^*\breve{\Pi}_V +O(\e|\log \e|^\gamma),
$$
with a uniform estimate for the remainder term. Here the truncated Poisson operator $\breve{\Pi}_V$ is introduced as $\breve{\Pi}_V:=\Pi_V\vert_{\breve{\mathcal H}}$ relative to the same truncated boundary space as above, $\breve {\mathcal H}=\P \mathcal H$.  As a result, we obtain
$$
R_\eff (z)-R_\hom(z)=O\bigl(\e|\log \e|^\gamma\bigr),
$$
with
\begin{equation}
\label{eq:Rhom}
R_\hom(z):= (A_0^E-z)^{-1}-\breve{S}_z^E\bigl(\breve{M}_E(z)+
 z \breve{\Pi}_V^*\breve{\Pi}_V\bigr)^{-1} (\breve{S}_{\bar z}^E)^*.
\end{equation}
By a classical result of \cite{Strauss} (see also \cite{Naimark1940,Naimark1943}), the operator $R_{\rm eff}(z)$
is a generalised resolvent, so it defines a
$z$-dependent family of closed densely defined operators in $L^2(Q_E),$ which are maximal anti-dissipative for $z\in \mathbb C_+$ and maximal dissipative for $z\in \mathbb C_-$. Writing the resolvent $(A_\e-z)^{-1}$ in the matrix form relative to the orthogonal decomposition $L^2(Q)=P_E L^2(Q)\oplus P_V L^2(Q)=L^2(Q_E)\oplus L^2(Q_V)$ then yields the following result.
\begin{theorem}
\label{thm:main_fin}
The resolvent $({A}_\e-z)^{-1}$ admits the following asymptotics in the uniform operator-norm topology:
$$
\bigl({A}_\e-z\bigr)^{-1}=\mathcal R_{\rm eff}(z) + O\bigl(\e|\log \e|^\gamma\bigr),
$$
where the operator $\mathcal R_{\rm eff}(z)$ has the following representation relative to the decomposition
$L^2(Q_E)\oplus L^2(Q_V)$:
\begin{equation}
	\label{eq:NRAbreve}
\mathcal R_{\rm eff}(z)=
\begin{pmatrix}
R_{\rm eff}(z)&\ \ \Bigl(\mathfrak K_{\bar z}\bigl[R_{\rm eff}(\bar z)-(A_0^{E}-\bar z)^{-1}\bigr]\Bigr)^*\breve{\Pi}_V^*\\[0.8em] \breve{\Pi}_V\mathfrak{K}_z \bigl[R_{\rm eff}(z)-(A_0^{E}-z)^{-1}\bigr] & \ \ \breve{\Pi}_V\mathfrak K_{z}\Bigl(\mathfrak K_{\bar z}\bigl[R_{\rm eff}(\bar z)-(A_0^{E}-\bar z)^{-1}\bigr]\Bigr)^*\breve{\Pi}_V^*
\end{pmatrix}.
\end{equation}
Here $\mathfrak{K}_z:=\Gamma_0^E|_{\mathfrak N_z}$ with $\mathfrak N_z:=\ran S_z^E\P$, $z\in \mathbb C_\pm$,
 and the generalised resolvent $R_{\rm eff}(z)$ is defined by \eqref{eq:Rhom}.
\end{theorem}

The above theorem provides us with the simplest possible leading-order term $\mathcal R_{\rm eff}(z)$ of the asymptotic expansion for $(A_\varepsilon-z)^{-1}.$ However, it is not yet obvious whether it
is the resolvent of some self-adjoint operator in the space $L^2(Q_E)\oplus \breve{\Pi}_V \breve{\mathcal H}\subset L^2(Q).$ It turns out that this is indeed so, which is seen via the following explicit construction.

Put $L^2(G):=\oplus_e L^2(0,l_e)$, $H^2(G):=\oplus_e H^2(0,l_e)$. For all $u\in H^2(G),$ denote by  $u_{ev}$ the limit of $u_e(x):=u|_e (x)$ at the vertex $v$.
Let $H_{\rm eff}:=L^2(G)\oplus \mathbb C^N,$ and set
\begin{equation}
\label{eq:domain_fin_modII}
\begin{aligned}
\dom \mathcal A_{\hom}=\bigl\{(u,\beta)^\top\in H_{\rm eff}:\ u\in H^2(G),\  &u_{ev}=u_{e'v}=:u_v \text{ for any } v  \\
&\text{ and } e,e' \text{ incident to } v, \text{ and }  \beta = \kappa u_V\bigr\},
\end{aligned}
\end{equation}
where $u_V$ is the $N$-dimensional vector of $\{u_v\}_{v\in V}$ and $\kappa$ is the diagonal matrix
\begin{equation}\label{eq:kappa}
\kappa:=\diag\bigl\{|Q_v^{0}|^{1/2}\bigr\}.
\end{equation}
The action of the operator is set by
\begin{equation}
\label{eq:operator_fin_modII}
\mathcal A_{\hom}\binom{u}{\beta}=
\left(\begin{array}{c}- u''\\[0.3em]
-\kappa^{-1} \partial_n u|_V
\end{array}\right),\qquad \binom{u}{\beta}\in \dom \mathcal A_{\hom},
\end{equation}
where $\partial_n u|_V$ is the $N$-dimensional vector $\{\sum_{e \sim v} \partial_n u_e|_v\}_{v\in V}$, i.e., the vector whose each element is represented by the sum of edge-inward normal derivatives of the function $u$ over all the edges incident to the vertex $v$. We write $e\sim v$ if and only if the edge $e$ is incident to the vertex $v$.

The main result of the present work, which is obtained by computing explicitly the resolvent of \eqref{eq:domain_fin_modII}--\eqref{eq:operator_fin_modII} and comparing it with \eqref{eq:NRAbreve} (see details of a similar computation in \cite{GrandePreuve}), is formulated next.
\begin{theorem}\label{thm:general_homo_result}
The resolvent $({A}_\e-z)^{-1}$ admits the following estimate in the uniform operator norm topology, uniform in $z\in K_\sigma$:
\begin{equation*}
\bigl({A}_\e-z\bigr)^{-1}-\Theta\bigl(\mathcal A_{\hom}-z\bigr)^{-1}\Theta^*=O\bigl(\e|\log \e|^\gamma\bigr),
\end{equation*}
where $\Theta$ is a partial isometry from  $H_{\rm eff}$ onto $L^2(Q)$, acting as follows:
\begin{itemize}
  \item For every edge $e\in G$, $e=[0,l_e]$, it embeds $u\in H^2(e)$ into $L^2(Q_e)$ as $u(x)\times \e^{-(d-1)/2} \pmb 1(y),$ where $y$ is the variable in the direction transverse to that of $x$;
  \item For every vertex $v\in G$, it embeds the value $u_v$, i.e., the common value of $u\in H^2(G)$ at the vertex $v$, into $L^2(Q_v)$ as $\e^{-(d-1)/2} u_v \pmb 1$.
\end{itemize}
\end{theorem}

\section{Analysis of vertex matching conditions}

In the present section we continue our study of the operator \eqref{eq:domain_fin_modII}--\eqref{eq:operator_fin_modII} associated with an arbitrary metric graph $G$, with a view to analysing its spectral structure. We will show that the matching conditions at graph vertices associated with the spectral problem for the mentioned operator, albeit closely resembling $\delta$-type conditions with coupling constants linear in the spectral parameter $z,$ can be in fact represented (up to a unitary gauge) by $\delta'$-type matching conditions at all $z\neq 0$ (while at $z=0$ they coincide with the classical Kirchhoff condition). We will assume throughout that this graph contains no loops, in line with the assumptions imposed on the thin network studied above. We will further assume without loss of generality that the graph $G$ is connected and that the matrix $\kappa$ is invertible.

Since the operator $\mathcal{A}_{\hom}$ can be viewed as a self-adjoint out-of-space extension (see, e.g., \cite{Strauss_survey} and references therein) of a symmetric differential operator on the metric graph $G$, it is amenable to the classical boundary triples theory, see \cite{Schmudgen,BHS}. We recall that for a closed and densely defined symmetric
operator $\mathcal A$ on a separable Hilbert space $H$ with domain~$\dom\mathcal A$,
a boundary triple is defined as follows.

\begin{definition}[\cite{Kochubej}]\label{def:bonudaryTriple}
	A triple $(\mathcal{K}, \Gamma_0, \Gamma_1 )$ consisting of an auxiliary
	Hilbert space $\mathcal{K}$ and linear mappings $\Gamma_0, \Gamma_1$ defined
	everywhere on  $\dom \mathcal A^*$ is called a \emph{boundary triple} for $\mathcal A^*$ if the following
	conditions are satisfied:
	\begin{enumerate}
		\item
		The abstract Green's formula is valid
		\begin{equation*}\label{GreenFormula}
			(\mathcal A^*\vec{u},\vec{v})_H - (\vec{u},\mathcal A^*\vec{v})_H = (\Gamma_1 \vec{u}, \Gamma_0 \vec{v})_{\mathcal{K}} -
			(\Gamma_0 \vec{u}, \Gamma_1 \vec{v})_{\mathcal{K}},\quad \vec{u},\vec{v} \in \dom\mathcal A^*
		\end{equation*}
		\item
		For any $Y_0, Y_1 \in{\mathcal{K}}$ there exist $\vec{u} \in
		\dom\mathcal A^*$, such that $\Gamma_0 \vec{u} = Y_0$, $\Gamma_1 \vec{u} = Y_1$.
		In other words,
		the mapping~$\vec{u} \mapsto \Gamma_0 \vec{u} \oplus \Gamma_1 \vec{u} $ from $\dom\mathcal A^*$
		to ${\mathcal{K}}\oplus {\mathcal{K}}$ is surjective.
	\end{enumerate}
\end{definition}
\noindent
It can be shown~(see \cite{Kochubej}) that a boundary
triple for~$\mathcal A^*$
exists,
%
although it is not unique.
%
%

\begin{definition}\label{def:WeylFunction}
	Let $\mathscr T = (\mathcal{K}, \Gamma_0, \Gamma_1 )$
	be a boundary triple
	of~$\mathcal A^*$.
	\emph{The Weyl function} of $\mathcal A^*$ corresponding to~$\mathscr T$ and
	denoted by $M(z)$, $z \in\mathbb C\setminus \mathbb R$, is an analytic
	operator-function with a positive imaginary part for~$z \in \mathbb C_+$
	(i.e., an operator $R$\nobreakdash-function) with values in
	the algebra of bounded operators on $\mathcal K$ such that
	\[
	M(z)\Gamma_0 \vec{u} = \Gamma_1 \vec{u} \qquad \forall\vec{u} \in {\rm ker}(\mathcal A^* -zI).
	\]
	For $z\in\mathbb C \setminus\mathbb R$ one has~$(M(z))^* = (M(\bar z))$ and
	$\Im (z)\Im(M(z)) > 0$.
\end{definition}

A comparison with the assertion 4 of Proposition \ref{prop:M} shows that the Weyl function $M(z)$ in the context of the boundary triples theory is intimately related to the object introduced in Definition \ref{defn:3}. The overall setup leading to its construction is, however, different and is based on the explicit choice of the boundary operators $\Gamma_0$ and $\Gamma_1$.

\begin{definition}
	An extension~$\mathscr A$ of a closed densely defined
	symmetric operator~$\mathcal A$ is
	called \emph{almost solvable}, and is denoted by $\mathscr A = A_B,$ if
	there exist a boundary
	triple~$(\mathcal{K}, \Gamma_0, \Gamma_1 )$ for~$\mathcal A^*$
	and a bounded operator~$B : \mathcal{K} \to \mathcal K$ defined
	everywhere in $\mathcal K$ such that
	\begin{equation*}
		\vec{u} \in \dom A_B \iff  \Gamma_1 \vec{u} =  B \Gamma_0 \vec{u}.
	\end{equation*}
\end{definition}
\noindent
This definition implies that $\dom A_B \subset \dom \mathcal A^*$
and $A_B$ is a restriction of~$\mathcal A^*$ to
the linear set $\dom A_B := \{\vec{u} \in \dom\mathcal A^* : \Gamma_1\vec{u} =
B \Gamma_0 \vec{u}\}$.
In this context, the operator~$B$ plays the r\^ole of a
parameter for the
family of extensions~$\{A_B \mid B : \mathcal K \to \mathcal K\}$.
It can be shown (see~\cite{CherKisSilva1} for references)
that
the resolvent set of~$A_B$ is
non-empty (i.e. $A_B$ is maximal),
both $A_B$ and $(A_B)^* = A_{B^*}$
are restrictions of $\mathcal A^*$ to their
domains, and
$A_B$ and $B$
are selfadjont or dissipative simultaneously.

Under the additional assumption that $\mathcal A$ is \emph{simple} (or, in other words, completely non-self-adjoint), that is, it has no reducing self-adjoint ``parts'',
the spectrum of~$A_B$ coincides, counting multiplicities, with the
set of points~$z_0 \in \mathbb C$ into which $(M(z_0) -B)^{-1}$
does not admit analytic continuation. In the general case, however, the spectrum is a union of the ``zeroes'' of the operator-valued function $M(z)-B$ introduced above and the spectrum of the self-adjoint ``part'' of the symmetric operator $\mathcal{A}$ in its Wold decomposition \cite{Krein_lectures}.

Our immediate aim is to construct a convenient boundary triple for the operator $\mathcal{A}_{\rm eff}$. In doing so, we rely upon the framework (in a particular case of a loop-graph with exactly one vertex) of the paper \cite{CENS}.

We define the symmetric operator $\mathcal{A}$ as follows (cf. \eqref{eq:domain_fin_modII}):
\begin{equation*}
	\begin{aligned}
	\dom \mathcal A=\Bigl\{(u,\beta)^\top\in H_{\rm eff}:\ &u\in H^2(G),\  u_{ev}=u_{e'v}=:u_v \text{ for any } v  \\
	&\text{ and } e,e' \text{ incident to } v,\quad \partial_n u|_V=0, \text{ and }  \beta = \kappa u_V \Bigr\},
	\end{aligned}
\end{equation*}
where $u_V$ is, as above, the $N$-dimensional vector of $\{u_v\}_{v\in V}$, $\kappa$ is the diagonal matrix \eqref{eq:kappa}, and $\partial_n u|_V$ is the $N$-dimensional vector of $\{\sum_{e \sim v} \partial_n u_e|_v\}_{v\in V}$. The action of the operator $\mathcal{A}$ is set by \eqref{eq:operator_fin_modII}. The operator thus defined is clearly symmetric in $L^2(G)\oplus\mathbb{C}^N$; its adjoint $\mathcal{A}^*$ is defined by the same expression \eqref{eq:operator_fin_modII} on the domain
\begin{equation*}
	\begin{aligned}
	\dom \mathcal A^*=\Bigl\{(u,\beta)^\top\in H_{\rm eff}:\ u\in H^2(G),\  u_{ev}=u_{e'v}=:u_v &\text{ for any } v  \\
	&\text{ and } e,e' \text{ incident to } v \Bigr\}.
	\end{aligned}
\end{equation*}

We have the following lemma.

\begin{lemma}\label{lemma:triple}
	Let $\mathcal{K}=\mathbb{C}^N$, $\Gamma_0 (u,\beta)^\top:= \partial_n u|_V$, and $\Gamma_1(u,\beta)^\top:= \kappa^{-1}\beta - u_V$. The triple $(\mathcal{K},\Gamma_0,\Gamma_1)$ is a boundary triple for the operator $\mathcal{A}^*.$ The operator $\mathcal{A}_{\rm eff}$ is a self-adjoint almost solvable extension of $\mathcal{A}$, corresponding to the matrix $B=0$ with respect to this boundary triple.
\end{lemma}

The \emph{proof} of the above lemma is obtained via integration by parts; for details, see \cite{CENS}.

We shall further require two ordinary differential operators on the metric graph $G$ together with their boundary triples. Consider $A^\delta_{\max}$ to be the operator generated by the negative Laplacian $-\Delta$ on $L^2(G)$, defined on the domain
$$
\dom A^\delta_{\max} = \{u\in H^2(G): u_v:=u_{ev}=u_{e'v} \text{ for any } v
\text{ and } e,e' \text{ incident to } v\}.
$$
In the paper \cite{YKK} it is shown that this is a natural choice of a maximal operator if one seeks to consider the so-called $\delta$-type matching conditions at the graph vertices, i.e., matching conditions of the type
\begin{equation}
u \text{ continuous at every vertex } v \text{ and } \partial_n u|_V = \eta^2 u_V,
\label{eta_Kirchhoff}
\end{equation}
where $\eta$ is a diagonal matrix. The conditions (\ref{eta_Kirchhoff}) reduce to Kirchhoff, or standard, matching conditions under the choice $\eta=0$. The natural boundary triple for $A^\delta_{\max}$ is $(\mathcal{K},\Gamma_0^\delta, \Gamma_1^\delta),$ where $\mathcal{K}=\mathcal{C}^N$,
$\Gamma_0^\delta u = u_V,$ and $\Gamma_1^\delta u = \partial_n u|_V$. The corresponding  Weyl function (``$M$-matrix") admits the form \cite{YKK} $M_\delta(z):=\{m^\delta_{vv'}\}_{v,v'\in V}$, where
\begin{equation}\label{eq:M_good}
	m^\delta_{vv'}(k)=\begin{cases}
		-k\sum_{e\sim v} \cot k l_e, & v=v',\\[0.2em]
		k\sum_{e\sim v, e\sim v'} (\sin k l_e)^{-1}, & v\not=v'.
	\end{cases}
\end{equation}
Here $k:=\sqrt{z}$ is such that $\Im k\geq 0.$

Newt, we consider the \emph{magnetic} Laplacian on the graph $G$ subject to $\delta'$-type matching at the vertices. Namely, we assume that for all edges $e$ the action of the operator is described as
\begin{equation}
-\biggl(\frac{d}{dx}+{\rm i}\tau_e\biggr)^2,
\label{delta_prime_op}
\end{equation}
where $\tau_e$ is an edgewise-constant magnetic potential. In order to introduce $\delta'$-type matching, consider the co-normal derivatives
$$
\partial_n^\tau u_e|_v =
\begin{cases}
	u_e'|_v+{\rm i}\tau_e u_e|_v & \mbox{if } v \text{ is the left endpoint of } e, \\[0.2em]
	-(u_e'|_v+{\rm i}\tau_e u_e|_v) & \mbox{otherwise}.
\end{cases}
$$
We say that the magnetic Laplacian on $G$ is subject to $\delta'$-type matching at the vertices if
\begin{equation}\label{eq:delta-prime-conditions}
	\partial_n^\tau u|_v:=\partial_n^\tau u_e|_v=\partial_n^\tau u_{e'}|_v \text{ for any } e,e'\sim v, \text{ and }
	\Sigma u|_V = \eta^2 \partial_n^\tau u|_V \quad \forall\ v\in V,
\end{equation}
where $\partial_n^\tau u|_V$ and  $\Sigma u|_V$ denote the vectors $\{\partial_n^\tau u|_v\}_{v\in V}$ and $\{\sum_{e\sim v} u_e|_v\}_{v\in V},$ respectively.

By an argument similar to that of \cite{YKK}, see also \cite{GrandePreuve}, one easily checks that $(\mathcal{K},\hat{\Gamma}_0,\hat{\Gamma}_1)$ is a boundary triple for the magnetic Laplacian $A^{\delta'}_{\max}$ defined on
$$
\dom A^{\delta'}_{\max} = \{u\in H^2(G): \partial_n^\tau u|_v:=\partial_n^\tau u_e|_v=\partial_n^\tau u_{e'}|_v  \text{ for any } v
\text{ and } e,e' \text{ incident to } v\},
$$
if  $\mathcal{K}=\mathbb{C}^N$, $\hat{\Gamma}_0 u:= \{\partial_n^\tau u|_v\}_{v\in V}$, and, finally,
$\hat{\Gamma}_1 u:= - \{\sum_{e\sim v} u_e|_v \}_{v\in V}$. The corresponding $M$-matrix  $\hat M(z):=\{\hat m_{vv'}\}_{v,v'\in V}$  admits the form \cite{YKK}
$$
\hat m_{vv'}(k)=\begin{cases}
	-k^{-1}\sum_{e\sim v} \cot k l_e, & v=v',\\
	-k^{-1}\sum_{e\sim v, e\sim v'} \exp({\rm i} \sigma_{e}(v,v')\tau_e l_e) (\sin k l_e)^{-1}, & v\not=v'.
\end{cases}
$$
Here $k=\sqrt{z}$ such that $\Im k\geq 0$ and $\sigma_e(v,v')=1$ if $e$ is directed from $v$ to $v'$, $\sigma_e(v,v')=-1$ otherwise.

We will now fix the values of the magnetic potential as follows: $\tau_e:=\pi/l_e$. The operator $A^{\delta'}_{\max}$ corresponding to this choice will be henceforth denoted by $\hat{A}_{\max}$. Its $M$-matrix $\hat M$ relative to the triple $(\mathcal{K},\hat{\Gamma}_0,\hat{\Gamma}_1)$ admits the form
$$
\hat m_{vv'}(k)=\begin{cases}
	-k^{-1}\sum_{e\sim v} \cot k l_e, & v=v',\\
	k^{-1}\sum_{e\sim v, e\sim v'} (\sin k l_e)^{-1}, & v\not=v',
\end{cases}
$$
which coincides with \eqref{eq:M_good} up to the factor $z^{-1}$. We remark that the operator $\hat{A}_{\max}$ and any of its self-adjoint restrictions can be unitary transformed into a regular (non-magnetic) Laplacian on the same graph $G$ by a standard gauge transform; this will however be reflected in the corresponding change of matching conditions at the vertices. Motivated by applications to electromagnetic waves propagation \cite{KisRyad}, here we prefer to proceed with the magnetic setup.

Returning to the analysis of the operator $\mathcal{A}_{\rm eff}$, we now have the following lemma,

\begin{lemma}\label{lemma:Mdeltadeltaprime}
	Relative to the boundary triple of Lemma \ref{lemma:triple}, the Weyl $M$-matrix of the operator $\mathcal{A}_{\rm eff}$ admits the form
	$$
	M(z)=-M_\delta^{-1}(z)-\frac{1}{z} \kappa^{-2}=-\frac{1}{z} (\hat M^{-1}(z)+\kappa^{-2}).
	$$
\end{lemma}

\begin{proof}
	In view of Lemma \ref{lemma:triple}, consider the vector $(u,\beta)^\top \in \dom \mathcal{A}^*$ such that
	$$
	\mathcal{A}^* \binom{u}{\beta}=z \binom{u}{\beta}.
	$$
	By the definition of $\mathcal{A}^*,$ one equivalently has
	$
	-u'' = zu \text{ on } G \text{ and } -\kappa^{-1} \partial_n u|_V =z \beta.
	$
	Abbreviating $\Gamma_0 (u,\beta)^\top =: \alpha$, one therefore has, in view of the definition of $\Gamma_0$ (see Lemma \ref{lemma:triple}):
	$$
	\beta=-\frac 1z \kappa^{-1} \alpha.
	$$
	Furthermore, taking now into account the definition of the second boundary operator $\Gamma_1$, one has
	$$
	\Gamma_1 (u,\beta)^\top = \kappa^{-1}\beta - u_V = -\frac 1z \kappa^{-2} \alpha -u_V.
	$$
	Note that the function $u$, by the definition of $\mathcal{A}^*$, must be continuous at every $v\in V$ and therefore belongs to the domain of the operator $A_{\max}^\delta$. Therefore, $u\in \ker (A_{\max}^\delta-z)$ and thus one has
	$
	\Gamma_1^\delta u = M_\delta(z) \Gamma_0^\delta u.
	$
	By construction, one has $\Gamma_0^\delta u = u|_V$ and $\Gamma_1^\delta u = \partial_n u|_V$, whence
	$
	u_V = M_\delta^{-1}(z) \alpha.
	$
	Ultimately,
	$$
	M(z)\alpha = \Gamma_1 (u,\beta)^\top = -\frac{1}{z} \kappa^{-2} \alpha - M_\delta^{-1}(z)\alpha,
	$$
	as claimed.
\end{proof}

Let $\mathcal{E}(\cdot)$ denote the orthogonal operator spectral measure of the self-adjoint operator $\mathcal{A}_{\rm eff}$ and $E(\cdot)$ the orthogonal spectral measure of the self-adjoint operator $\hat{A}_{-\kappa^2}$, where the latter is defined as the almost solvable extension of $(A^{\delta'}_{\max})^*$ corresponding to the parameterising operator $B=-\kappa^2$. In other words, it is the magnetic Laplacian on the graph $G$ subject to the condition $\tau_e=\pi/l_e$ for all $e\in E$ and $\delta'$-type matching conditions \eqref{eq:delta-prime-conditions} at the graph vertices with $\eta=\kappa$. We have the following statement.

\begin{theorem}\label{thm:last}
	For any $c_0>0$ the operators $\mathcal{A}_{\rm eff} \mathcal{E}(c_0,\infty)$ and $\hat A_{-\kappa^2} E(c_0,\infty)$ are unitary equivalent.
\end{theorem}

\begin{proof}
	In the generic case when $\mathcal{A}$ is simple (in particular when all $l_e$ are rationally independent), the claim follows immediately from Lemma \ref{lemma:Mdeltadeltaprime}. Indeed, in this case no positive $z$ can be an eigenvalue of the operator $\mathcal{A}$ and the same applies to the operator $(A^{\delta'}_{\max})^*$. Therefore, any reducing self-adjoint ``part'' of either symmetric operator can only be zero.
	
	In the general case, if $z_0>0$ is an eigenvalue of $\mathcal{A}$, then the corresponding eigenfunction solves
	$-u'' = z_0u$ on $G$
	subject to $-\kappa^{-1} \partial_n u|_V = z_0 \beta,$ $u_V=\kappa^{-1} \beta,$ $\partial_n u|_V =0$
	and therefore $u_V=\partial_n u|_V=0$. Thus any eigenfunction must be of the form $c_e \sin \sqrt{z_0} x$ on each $e$; moreover, one must also have $\sin \sqrt{z_0} l_e =0$ for all $e$ such that $c_e\not = 0$ and $\partial_n u|_V =0$. For each of these, the function that is edgewise transformed as $c_e \sin \sqrt{z_0} x \mapsto c_e \exp(-{\rm i}\tau_e x)\cos \sqrt{z_0} x$ is shown to be an eigenfunction of $(A^{\delta'}_{\max})^*$ corresponding to the eigenvalue $z_0$. The same argument applied in the opposite direction completes the proof.
\end{proof}

\begin{remark}
	The analysis of the unitary equivalence linking  $\mathcal{A}_{\rm eff}\mathcal{E}(c_0,\infty)$ and $\hat A_{-\kappa^2} E(c_0,\infty)$ is an exciting possible development from the point of view of classical functional analysis, since it appears to be a natural graph-based generalisation of the classical Hilbert transform. This can be seen, in particular, from the explicit calculation in the case of an infinite chain graph \cite{CherKis}.
\end{remark}

The above theorem shows that the spectral analysis of the operator $\mathcal{A}_{\rm eff}$ in relation to its non-zero spectrum reduces to that of the magnetic graph Laplacian with (non-trivial) $\delta'$-type matching condition at the graph vertices, which could come as a surprise given that the eigenvalue problem for $\mathcal{A}_{\rm eff}$ yields, for the first component of the eigenvector $(u,\beta)^\top,$ the equation
$-u'' = zu$ on $G$ subject to
$\partial_n u|_V = - z \kappa^2 u|_V,$
which on the face of it is a $\delta$-type matching condition, albeit with coupling constants proportional to the spectral parameter $z$. It would seem natural, therefore, for the self-adjoint operator $\mathcal{A}_{\rm eff}$ to be a $H^{-1}$ singular perturbation of the graph Laplacian with standard boundary conditions. Instead, our result shows that it is in fact a more singular $H^{-2}$ perturbation. The underlying reasons of this peculiar behaviour are discussed in detail in \cite{CEKS2022}, where the relationship of operators of the class considered with those in the area of zero-range potentials with internal structure, as introduced by B.\,S.\,Pavlov \cite{Pavlov1987, Pavlov_internal_structure}, is explained. We also point out that the assertion of Theorem \ref{thm:last} has been observed in a particular case of the cycle graph with one vertex (the loop) in \cite{CherKis} in the context of a high-contrast homogenisation problem on the real line.

We conclude the spectral analysis of the operator $\mathcal{A}_{\rm eff}$ by a brief discussion of its kernel and the comparison of the latter with that of the operator $\hat A_{-\kappa^2}$. It turns out that unlike what happens with its non-zero spectrum, the kernel $\ker \mathcal{A}_{\rm eff}$ of the operator $\mathcal{A}_{\rm eff}$ is that of the Kirchhoff graph Laplacian.

Note first that $\ker \mathcal{A}_{\rm eff}$ necessarily belongs to the non-simple (i.e., self-adjoint) part of $\mathcal{A}$. Indeed, for $(u,\beta)^\top\in\ker \mathcal{A}_{\rm eff}$ it has to satisfy
$$
-u'' = 0 \text{ on } G,\qquad
\partial_n u|_V =0, \quad u_V=\kappa^{-1} \beta,
$$
where the boundary conditions are equivalent to $\Gamma_0 (u,\beta)^\top=0$, $\Gamma_1 (u,\beta)^\top=0$. On the other hand, this is precisely the condition for $u$ to be in the kernel of a graph Laplacian with Kirchhoff matching conditions. One therefore infers from \cite{Kuchment2} that $\dim\ker \mathcal{A}_{\rm eff}=1$ (in our case of connected graphs), and the elements of $\ker \mathcal{A}_{\rm eff}$ are constants on $G$.

The kernel of $\hat A_{-\kappa^2}$ is spanned by functions $u$ such that
$$
-u''=0 \text{ on } G \text{ and } \Sigma u|_V=\kappa^2\partial^\tau_n u|_V.
$$
It is easily checked that a non-trivial solution to this problem could exist only if $\partial^\tau_n u|_V=0,$ $\Sigma u|_V=0$ or, in other words, if $\hat \Gamma_0 u = \hat \Gamma_1 u =0$. This means that,  precisely as in the case of $\mathcal{A}_{\rm eff}$, the kernel of $\hat A_{-\kappa^2}$ necessarily belongs to the non-simple (i.e., self-adjoint) part of the symmetric operator $(A^{\delta'}_{\max})^*$.
The question of its existence and dimension admits a simple answer in terms of the graph topology. It is clear that it is trivial in the case when $G$ is a tree; in general its dimension is shown to be equal to the cyclomatic number $\chi$ of the graph $G$. In particular, this yields unitary equivalence of $\mathcal{A}_{\rm eff}$ and $\hat A_{-\kappa^2}$ in the case where $G$ contains exactly one cycle.

\begin{remark}
	The result of the present section seems to have been overlooked in a number of now-classical papers dealing with Sturm-Liouville problems on an interval with boundary conditions depending on a spectral parameter,  see e.g. \cite{Evans, Schneider, Walter, Fulton, Hinton, Shkalikov_1983}.
\end{remark}

\begin{conjecture}
The above discussion raises the question of which definition of $\delta'$-type interaction on a graph is motivated physically, i.e. whether it is the one emerging from the analysis of thin networks as the operator ${\mathcal A}_{\rm hom},$ see \eqref{eq:domain_fin_modII}--\eqref{eq:operator_fin_modII}, or the traditional (see \cite{Kuchment2}) definition \eqref{delta_prime_op}--\eqref{eq:delta-prime-conditions}. At first sight, the difference between the two operators is insignificant: it is only in their kernels. However, it can happen to be of paramount importance if, e.g., one considers an $\varepsilon$-periodic graph with $\delta'$-type matching conditions, in which case the homogenisation procedure \cite{Physics, ChEK} will lead to drastically different outcomes for the two related setups, as it relies upon a ``threshold effect" \cite{BirmanSuslina} in the behaviour of the least eigenvalue of the operator on the fundamental cell for small quasimomenta.
\end{conjecture}

\section*{Acknowledgements}

KDC, YYE are grateful for the financial support of EPSRC Grants EP/L018802/2.
KDC, AVK are grateful for the financial support of EPSRC Grants
EP/V013025/1. YYE, AVK are grateful to IIMAS--UNAM for the hospitality and financial support during the research visit when part of this work was carried out. 

\emph{Data access statement.} No new data were generated or analysed during this study.

\end{document}